\documentclass[a4paper,10pt,reqno,norelsize]{amsart}

\usepackage{hyperref}
\usepackage{verbatim} 

\usepackage{graphicx}
\usepackage{caption}
\usepackage{subcaption}

\usepackage[foot]{amsaddr}

\usepackage[ruled,vlined]{algorithm2e}

\newcommand{\abs}[1]{\left\vert #1 \right\vert}
\newcommand{\norm}[1]{\left\Vert #1 \right\Vert}
\newcommand{\set}[1]{\left\lbrace #1\right\rbrace}
\newcommand{\sse}{\subseteq}
\newcommand{\sprod}[1]{\left\langle #1 \right\rangle}

\newcommand{\cC}{\mathcal{C}}
\newcommand{\cM}{{\mathcal M}}
\newcommand{\cA}{\mathcal{A}}
\newcommand{\cH}{{\mathcal H}}
\newcommand{\cK}{{\mathcal K}}

\newcommand{\R}{{\mathbb R}}
\newcommand{\C}{{\mathbb C}}

\newcommand{\Z}{{\mathbb Z}}

\newcommand{\bT}{\mathbb{T}}

\newcommand{\bH}{{\mathbb H}}
\newcommand{\bK}{{\mathbb K}}

\newcommand{\dZ}{\Z}

\newcommand{\tr}{\text{tr}}

\usepackage{dsfont}

\DeclareMathOperator{\supp}{supp} 

\DeclareMathOperator{\Span}{span}
\DeclareMathOperator{\m}{mod}

\newtheorem{proposition}{Proposition}
\newtheorem{theorem}{Theorem}
\newtheorem{lemma}{Lemma}

\theoremstyle{definition}
\newtheorem{definition}{Definition}
\newtheorem{example}{Example}
\newtheorem{remark}{Remark}

\numberwithin{equation}{section} 
\numberwithin{proposition}{section}
\numberwithin{theorem}{section}
\numberwithin{definition}{section}
\numberwithin{example}{section}
\numberwithin{remark}{section}
\numberwithin{lemma}{section}

\begin{document}
\title{Phase Retrieval from Gabor Measurements}
\author{Irena Bojarovska$^1$ and Axel Flinth$^1$}
\address{$^1$Institut f\"ur Mathematik, Technische Universit\"at Berlin, Germany}
\email{bojarovska@math.tu-berlin.de, flinth@math.tu-berlin.de}
\thanks{The final publication is available at Springer via \href{ http://dx.doi.org/10.1007/s00041-015-9431-0}{http://dx.doi.org/10.1007/s00041-015-9431-0}.}
\begin{abstract}
Compressed sensing investigates the recovery of sparse signals from
linear measurements. But often, in a wide range of applications, one is given only
the absolute values (squared) of the linear measurements. Recovering such signals (not necessarily sparse) is known as the phase retrieval problem. We consider this problem in the case when the measurements are time-frequency shifts
of a suitably chosen generator, i.e. coming from a Gabor frame. We prove an easily checkable injectivity condition
for recovery of any signal from all $N^2$ time-frequency shifts, and for recovery of sparse signals, when only
some of those measurements are given.


\end{abstract}

\maketitle
\begin{section}{Introduction}
Phase retrieval, a common problem in a variety of applications including X-Ray crystallography, optical imaging and electron microscopy, is the task of recovering a signal from the squares of the absolute values of its linear measurements. The best one can hope for in this case is to recover the signal up to a unimodular constant, because $x$ and $cx,$ where $\vert c \vert =1$ will always give the same measurements. To fix notation, let $F = \left( f_i \right)_{i=1}^m \subseteq \bK^N$ be a set of measurement vectors, where $\bK$ is $\R$ or $\C.$ Further, let $\bT = \{ c \in \bK : \vert c \vert =1\}.$ The measurement process is then given by the map
$$\cM_F: \bK^N \slash \bT  \rightarrow \R_+^m, \quad \cM_F(x) = \begin{bmatrix}  \vert \langle x, f_1 \rangle \vert^2   & \vert  \langle x, f_2 \rangle \vert^2  &\ldots& \vert  \langle x, f_m \rangle \vert^2  \end{bmatrix}^T.$$
The task is to recover $x$ up to a global phase, given $\cM_F(x).$  We say that $F$ \textit{allows phase retrieval}, if the map $\cM_F$ is injective. 

There are three main directions of questions that one is interested in when looking into this problem:
\begin{itemize}
\item{Injectivity:}  Which properties of the measurement vectors can give us necessary and/or sufficient conditions on the injectivity of the map $\cM_F?$
\item{Minimal number of measurements:} How many measurements are needed for a set $F$ to allow phase retrieval?
\item{Algorithms:} How can one practically find $x,$ given the intensity measurements $\cM_F(x)$? 
\end{itemize}
Comprehensive answers to these questions and open problems can be found in \cite{balan2006signal, candes2013phase, alexeev2014phase, bandeira2014saving}. Motivated by different applications, these general questions can also be asked for only a particular type of measurements, and/or signals. In this paper we focus on Gabor, or differently said, short-time Fourier measurements, which are time-frequency shifts of a suitably chosen generator. This type of measurements is of particular interest for many applications in speech and audio processing \cite{nawab1983signal}, ptychographical CDI \cite{guizar2008phase} etc. On the signal side, we are looking at the sparsity constraint -- a natural assumption that the 
signal we want to recover has only few non-zero entries, or is a linear combination of few vectors (has sparse representation). 
The sparsity constraint is a novel paradigm for signal- and image processing, and utilized, in particular, for compressed sensing methodologies \cite{eldar2012compressed}. 
Sparse phase retrieval is studied in \cite{ohlsson2014conditions, wang2014phase}.

A combination of phase retrieval from Gabor measurements for sparse signals was firstly considered in \cite{eldar2015sparse}. The theoretical results there are about the recovery of non-vanishing signals from a full set of $N^2$ Gabor measurements, and some intuition about the difficulty of recovery of sparse signals is given. Numerical results show that recovering sparse signals can be effectively conducted with modification of the GESPAR algorithm \cite{shechtman2014gespar}, using less than $N^2$ measurements. 

In the very recent work \cite{iwen2015fast}, both theoretical and numerical investigations show, that $O(N\log^3(N))$ measurements are enough for recovering general signals from block circulant Fourier based measurements, and if the signal is $k$-sparse, only $O(k \log^5(N))$ measurements. The structure of the measurements is similar to the one of Gabor systems, but at this moment it is not clear how their results transfer to the Gabor setting that we consider here.

In this paper, we investigate the problem of phase retrieval from Gabor measurements, for both full and sparse signals. 
Our main concern is the question of injectivity. Using the characterization of phase retrievability via the properties of the kernel of the PhaseLift operator \cite{bandeira2014saving}, we provide a condition on the generator, sufficient for the corresponding Gabor system to allow phase retrieval. We show how this condition can be eased, if the signal that needs to be recovered is non-vanishing. We further provide two examplary classes of generators, complex random signals and characteristic functions of difference sets, which satisfy the above mentioned condition. The common Gabor generators, which are short windows or Alltop sequences, on the other hand -- as we will show -- are not suitable 
for phase retrieval of general signals (they fail to recover sparse signals), this problem was also considered in \cite{iwen2015fast}.

Further, we extend the injectivity condition from  \cite{bandeira2014saving} to the sparse setting, and provide a similar, but more involved condition on the generator which can guarantee us phase retrievability of sparse signals, additionally with less than $N^2$ measurements. We generalize this result also to signals which are sparse in a Fourier domain. When $N$ is prime, we construct generators such that the Gabor system can do $k$-sparse phase retrieval from $O(k^3)$ measurements. As we will see, the result can also be interpreted as an injectivity condition for recovery of structured $k^2$-sparse vectors from $O(k^3)$ \textit{linear} measurements.

Both the injectivity theorems naturally provide a simple algorithm for recovery of signals up to a global phase. When all $N^2$ measurements are given, the recovery of any signal is possible by using solely the fast Fourier transform, making the algorithm extremely fast. If some of the measurements are lost, we can employ $\ell_1$ minimization to get the signal back. We provide several numerical experiments to test this idea for various settings. Although the number of measurements for recovery in our work is still of relatively high order, the ideas we use are novel, and might be of interest for the community for better understanding of this problem, and its future development.

The remainder of this paper is organized as follows: Section \ref{sec:arbitrary} is dedicated to the phase retrievability question 
for general signals, from all $N^2$ Gabor measurements. In Section \ref{sec:sparse}, we focus on the sparse setting, and show that $k$-sparse phase retrieval is possible with order of $k^3$ Gabor measurements.
A detailed description of the algorithm that we propose, and its empirical evaluation is presented in Section \ref{sec:algorithm}. 

\end{section}

\begin{section}{An injectivity condition for arbitrary signals} \label{sec:arbitrary}
\begin{subsection}{Notation and basic objects}
We will be working in the signal space $\C^N,$ as a space of complex valued, $N$ periodic functions with integer 
argument, $x=x(j),$ $j \in \Z,$ which therefore always has to be assumed modulo $N.$ We will use the customary domain $[0,\ldots,N-1]$ of $j,$ but we will often write $j\in \Z_N$ for convenience. The scalar product between two signals $x$ and $y$ is defined as
$$\sprod{x,y} = \sum_{j=0}^{N-1} \bar{x}(j) y(j),$$
and the Hilbert Schmidt scalar product between two $N\times N$ matrices $A$ and $B$ as
	$$\sprod{A,B}_{HS} = \tr(A^*B)= \sum_{i \in \Z_N} \sprod{A e_i,B e_i}. $$
The $N$-th root of unity will be denoted by $\omega = e^{\frac{2\pi i}{N}}.$ We define the (discrete) Fourier transform $\hat{x}$ and the inverse Fourier transform $\check{x}$ of $x \in \C^N$ as follows:
\begin{align*}
	\hat{x}(j)= \sum_{n=0}^{N-1} x(n) \omega^{-nj}, \quad \check{x}(j) = \frac{1}{N}\sum_{n=0}^{N-1} x(n) \omega^{nj}.
\end{align*}
 A family of vectors $\left( \phi_i \right)_{i=1}^M$ in $\C^N$ is called a \textit{finite frame} for $\C^N$ \cite{casazza2012finite}, if there exist
constants $0 < A \leq B < \infty$ such that 
\begin{equation*} \label{eq:deff}
 A \Vert x \Vert^2 \leq \sum_{i=1}^M \vert \langle x,\phi_i \rangle \vert^2 \leq B \Vert x \Vert^2 \quad \text{for all } x \in \C^N.
\end{equation*}
If $A=B$ is possible, then $\left( \phi_i \right)_{i=1}^M$  is called an $A$-\textit{tight frame}.

For $p \in \Z_N$, we define the translation operator $T_p: \C^N \mapsto \C^N$ through
\begin{align*}
	(T_p x)(n)= x(n-p)
\end{align*}
Further, we define for $\ell \in \Z_N$, the modulation operator $M_\ell:\C^N \mapsto \C^N$ through
\begin{align*} 
	(M_\ell x)(n) = \omega^{\ell n} x(n).
\end{align*}
A \textit{Gabor frame} \footnote{We in fact always obtain a frame in this way if $g\neq 0$. The frame is even $N \norm{g}^2$-tight.\cite{pfander2013gabor} } is the collection of all translations and modulations of a single vector $g \in \C^N,$ 
$$\left( M_l T_p g \right)_{l,p=0}^{N-1}.$$
For a pair $\lambda=(p, \ell)$ sometimes we will use the short-hand notation $\Pi_\lambda := M_\ell T_p,$ and $g_\lambda:= \Pi_\lambda g.$  The matrices of those operators are unitary, and the collection of them forms a basis of $\C^{N \times N}$ \cite{pfander2013gabor}, i.e.
\begin{align} \label{eq:orth}
	\sprod{\Pi_\lambda, \Pi_\mu}_{HS} = N \delta_{\mu, \lambda}.
\end{align}
We will need the following well-known commutation relations between translations and modulations.
\begin{lemma}\cite{groechenig01foundations} \label{lem:commutationRelation}
Let $\lambda=(p, \ell), \mu=(q,j) \in \Z_N^2$. Then, we have
\begin{align*}
	M_\ell T_p &= \omega^{\ell p} \, T_p M_\ell, \\
	\Pi_{\lambda}\Pi_{\mu} &= \omega^{-jp}\omega^{\ell q} \Pi_\mu\Pi_\lambda.
\end{align*}
\end{lemma}

\end{subsection}

\begin{subsection}{Injectivity for full Gabor measurements}
As mentioned before, we want to pose the question under what conditions a signal $x$ from some class $\mathcal{C} \sse \C^N$ can be recovered from a set of its \textit{Gabor intensity measurements} $(\abs{\sprod{x, g_\lambda}}^2)_{\lambda \in \Lambda},$ $\Lambda \sse \Z_N^2.$ Since these measurements are invariant under multiplication with $c \in \bT = \set{c \in \C, \abs{c}=1}$, the best we can hope for is to recover $x$ up to a global phase. If we denote by~$\C^N\slash\bT$ the set of equivalence classes under the equivalence relation $x \sim y \Leftrightarrow \exists c \in \bT : x= cy$, we can formally pose the problem as follows: Under what conditions on $g$ is the map
\begin{align*}
\mathcal{M}_G : \mathcal{C} \slash \bT \to \R^{\abs{\Lambda}}_+, \quad x \mapsto (\abs{\sprod{x, g_\lambda}}^2)_{\lambda \in \Lambda}
\end{align*}
injective?
\begin{definition}
We say that the Gabor system $G=\left( g_\lambda \right)_{\lambda\in \Lambda}$ associated to a generator $g \in \C^N$ is \emph{allowing phase retrieval for $\mathcal{C}$} (or has the \emph{phase retrieval property}), if the map $\cM_G$ is injective.
\end{definition}

We start by considering the problem of recovering arbitrary signals from all measurements, i.e. $\mathcal{C}=\C^N$ and $\Lambda= \Z_N^2$. In order to investigate which Gabor frames are allowing phase retrieval for this class, we will use a  well known characterization of the phase retrieval property in the complex case, given via the properties of the kernel of the \textit{PhaseLift} operator, also called \textit{super analysis operator} in \cite{bandeira2014saving}.  For a set of measurement vectors $\left(f_i\right)_{i=1}^m$ in $\C^N$ this operator is defined as
	\begin{align} \label{eq:opA}
 		\cA : \C^{N\times N} \to \C^m, \quad H \mapsto \left(\sprod{H, f_i f_i^*}_{HS}\right)_{i=1}^m,
 	\end{align}
 Note that when $H$ is in the form $xx^*,$ $ \sprod{H, f_i f_i^*}_{HS} =\sprod{f_i, Hf_i} = \vert \langle x,f_i \rangle \vert^2.$ Also note that the authors of \cite{bandeira2014saving} chose the set of Hermitian matrices as the domain of $\mathcal{A}$. We define $\cA$ in this way in order to avoid some technicalities. However, the space of Hermitian matrices is a very natural domain in the context of phase retrieval, as the next theorem suggests. 
\begin{theorem}\cite{bandeira2014saving} \label{th:PRfull}
 	A set of measurement vectors $\left(f_i \right)_{i=1}^m$ allows phase retrieval if and only if 
 	 	the kernel of the associated map $\cA$ does not contain any Hermitian matrices of rank 1 or 2.
\end{theorem}
With this theorem, we can prove that the full set of  $N^2$ Gabor intensity measurements allows phase retrieval, as long as a simple condition is satisfied.
\begin{theorem} \label{theorem:GaborCond}
	Let $g \in \C^N$ be a generator for which 
	\begin{align} \label{eq:GaborCond}
		\sprod{g, g_\lambda} \neq  0
	\end{align}
	for every $\lambda \in \Z_N^2$. Then the corresponding Gabor frame $G = \left( g_\lambda \right)_{\lambda\in \Z_N^2}$ allows phase retrieval.
\end{theorem}
\begin{proof}
Theorem \ref{th:PRfull} suggests that we should investigate $\sprod{g_\lambda, H g_\lambda}$ for $H \in \bH^{N \times N}$. Equality \eqref{eq:orth} implies that
\begin{align*}
	H= \frac{1}{N}\sum_{\mu\in \Z_N^2} \sprod{\Pi_\mu,H}_{HS} \Pi_\mu.
\end{align*}
If $\mu = (p, \ell),$ we have
\begin{align*}
	\sprod{\Pi_\mu,H}_{HS}= \sum_{i \in \Z_N} \sprod{\Pi_\mu e_i, He_i} = \sum_{i \in \Z_N} \sprod{\omega^{\ell(i+p)} e_{i+p}, He_i}= \sum_{i \in \Z_N}\omega^{-\ell i} \sprod{ e_{i}, He_{i-p}} = \widehat{\cH}_p(\ell),
\end{align*}
where $\widehat{\cH}_p$ denotes the (discrete) Fourier transform of the vector $\cH_p$, defined by $\cH_p(i)=H_{i,i-p}$. Note that $\cH_p$ is in some sense the $p$-th 'band' of the matrix $H$.
It hence holds
\begin{align*}
	N \sprod{H, g_\lambda g_\lambda^*}_{HS} = N\sprod{g_\lambda, H g_\lambda} = \sum_{p, \ell} \sprod{g_\lambda,\widehat{\cH}_p(\ell)\Pi_{(p, \ell)}g_\lambda} = \sum_{p, \ell} \widehat{\cH}_p(\ell)\sprod{g_\lambda,\Pi_{(p, \ell)}g_\lambda}.
\end{align*}
If we write $\lambda = (q, j)$, we know by Lemma \ref{lem:commutationRelation} that $\Pi_{(p, \ell)}g_\lambda =\Pi_{(p, \ell)}\Pi_{(q, j)}g = \omega^{-jp} \omega^{\ell q} \Pi_{(q,j)}\Pi_{(p, \ell)}g$. Using this, and the fact that $\Pi_\lambda$ is unitary, we arrive at
\begin{align} \label{eq:GaborPLisFFT}
	N\sprod{g_\lambda, H g_\lambda}= \sum_{p, \ell}\omega^{-jp}\omega^{\ell q} \widehat{\cH}_p(\ell)\sprod{g,g_{p, \ell}}.
\end{align}
Now, assume that this vanishes for all $\lambda= (q, j) \in \Z_N^2$. Fixing $j$, we see that the above expression is just the value of the Fourier transform  of the vector $V^q \in \C^N$ with $p$th entry
\begin{align*}
	V^q(p) = \sum_{ \ell}\omega^{\ell q} \widehat{\cH}_p(\ell)\sprod{g,g_{p, \ell}}
\end{align*}
evaluated at $j$. Since \eqref{eq:GaborPLisFFT} equals zero for all $j$, the vector $V^q$ vanishes for every $q$. Further, we observe that $V^q(p)$ is $N$ times the value at $q$ of the inverse Fourier transform of the vector $w^p \in \C^N,$  where
\begin{align} \label{eq: GaborMultiplier}
	w^p(l) = \widehat{\cH}_p(\ell)\sprod{g,g_{p, \ell}}.
\end{align}
This expression must therefore be equal to zero for all $p$ and $\ell$. With the assumption on the generator, we conclude that all the vectors $\widehat{\cH}_p$ must vanish, and therefore also $H$. $H$ can hence not have rank $1$ or $2$, and the proof is finished.
\end{proof}

	Carefully going through the argument of the last proof, we see that it shows that the only matrix in the kernel of $\cA$ is the zero matrix. 
	Therefore, the proof actually shows that, under the assumption \eqref{eq:GaborCond}, $\cA$ is an injective map. We use this idea to prove the following theorem.
	\begin{theorem} \label{th:framebounds}
		Let $g\in \C^N$ be such that $\sprod{g, g_\lambda}\neq0$ for all $\lambda \in \Z_N^2$. Then, the $N^2$ rank-1 operators $\left(g_\lambda g_\lambda^*\right)_{\lambda \in \Z_N^2}$ form a frame for $\C^{N \times N}$ (equipped with the Hilbert-Schmidt norm) and hence a basis. The frame bounds are given by
		\begin{align*}
			A= N \cdot \min_{\lambda \in \Z^2} \abs{\sprod{g, g_\lambda}}^{2}, \quad B = N \cdot \max_{\lambda \in \Z^2} \abs{\sprod{g, g_\lambda}}^{2}.
		\end{align*}
	\end{theorem}
	\begin{proof}
		What we need to prove that for every $H \in \C^{N \times N}$
	\begin{align*}
		A \norm{H}_{HS}^2 \leq \sum_{\lambda \in \Z^2} \abs{\sprod{H,g_\lambda g_\lambda^*}_{HS}}^2 \leq B \norm{H}_{HS}^2.
	\end{align*}	
 In other words, we need to prove that $A\norm{H}_{HS}^2 \leq \norm{\cA(H)}^2 \leq B \norm{H}_{HS}^2$, where $\cA$ is the PhaseLift operator \eqref{eq:opA}. Using the notation of the proof of Theorem \ref{theorem:GaborCond}, the formula \eqref{eq:GaborPLisFFT} states that the $N$-tuple $(V^q)_{q=1}^N \in (\C^N)^N$ is obtained by performing inverse Fourier transforms of the columns of the matrix  $(N \! \sprod{H,g_\lambda g_\lambda^*}_{HS})_{\lambda \in \Z_N^2}=N\cA(H)$. Hence, their norms are related as follows:
\begin{align*}
	\norm{N \cA(H)}^2 = \Vert (\hat{V}^q)_{q=1}^N \Vert^2 = N\norm{(V^q)_{q=1}^N}^2.
\end{align*}
Using the same argument, we obtain
\begin{align*}
	\norm{(V^q)_{q=1}^N}^2 = \norm{(N\check{w}^p)_{p=1}^N}^2 = \frac{N^2}{N}\norm{(w^p)_{p=1}^N}^2 \text{ and } \norm{ (\widehat{\cH}^p)_{p=1}^N}^2 = N \norm{(\cH^p)_{p=1}^N}^2.
\end{align*}
The $N$-tuples $(\widehat{\cH}^p)_{p=1}^N$ and $(w^p)_{p=1}^N$ are related through \eqref{eq: GaborMultiplier}. Therefore, if we define $\alpha =\min_{\lambda \in \Z^2} \abs{\sprod{g, g_\lambda}}$, $\beta = \max_{\lambda \in \Z^2} \abs{\sprod{g, g_\lambda}}$, we have
\begin{align*}
	\norm{(w^p)_{p=1}^N}^2  = \sum_{ (p,\ell) \in \Z_N^2} \abs{w^p(\ell)}^2 = \sum_{ (p,\ell) \in \Z_N^2}  \abs{\widehat{\cH}_p(\ell)\sprod{g,g_{p, \ell}}}^2 \ \ \begin{matrix} \leq \beta^2 \norm{ (\widehat{\cH}^p)_{p=1}^N}^2 \\ 
		\geq \alpha^2 \norm{ (\widehat{\cH}^p)_{p=1}^N}^2 \end{matrix}.
\end{align*}
Finally, the matrix $H$ is obtained by merely permuting the elements of the array $(\cH^p)_{p=1}^N$. Hence $\norm{H}_{HS}^2 = \norm{(\cH^p)_{p=1}^N}^2$. Combining everything, we obtain
\begin{align*}
\norm{\cA(H)}_{HS}^2 = \frac{N}{N^2}\norm{(V^q)_{q=1}^N}^2 = \norm{(w^p)_{p=1}^N}^2 \begin{matrix}&\leq \beta^2 \norm{ (\widehat{\cH}^p)_{p=1}^N}^2 &= N \beta^2 \norm{(\cH^p)_{p=1}^N}^2 &= N \beta^2 \norm{H}_{HS}^2\\ 
		&\geq \alpha^2 \norm{(\widehat{\cH}^p)_{p=1}^N}^2 &= N \alpha^2 \norm{(\cH^p)_{p=1}^N}^2 &= N \alpha^2 \norm{H}_{HS}^2\end{matrix} ,
\end{align*}
which is exactly what we wanted to prove.

	\end{proof}

\subsubsection{Non-vanishing vectors}
We will now show that if we are interested in recovery of only non-vanishing vectors, weaker conditions on the generator 
can be assumed. 
\begin{definition}
	A vector $x \in \C^N$ is called \textit{non-vanishing} (or \textit{full}), if all its entries are nonzero, i.e. 
	$$x(n) \neq 0, \quad \text{ for all } n=0,\ldots, N-1.$$
	By $\cC_f$ we denote the class of all non-vanishing signals in $\C^N$.
\end{definition}
 
This situation is much easier to handle, because, intuitively, the non-presence of ''holes'' in the signals keeps the phases of the entries coupled.  We will use the same technique as in Theorem \ref{theorem:GaborCond} to prove that the injectivity condition can be weakened in this setting.  Note that we are still assuming that all measurements are known.
\begin{theorem}
	 Assume  that
	 \begin{equation} 
	      \sprod{g, g_{p,\ell}}\neq 0 \text{ for } p=0,1 \text{ and } \ell \in \Z_N \label{eq:weakCond}. 
	 \end{equation} 
	 Then the Gabor frame $G=\left( g_\lambda \right)_{\lambda \in \Z_N^2}$ allows phase retrieval for $\cC_{f}.$
\end{theorem} 
\begin{proof} Assume that \eqref{eq:weakCond} is satisfied, and that $x$ and $y$ are full vectors which are measured equally by the Gabor frame. Then $H:=xx^*-yy^*$  is in the kernel of $\cA$, since for every $\lambda \in \Z_N^2$ we have
\begin{align*}
\cA(xx^*-yy^*)(\lambda) = \sprod{g_\lambda, (xx^* - yy^*)g_\lambda}=  \abs{\sprod{g_\lambda,x}}^2-\abs{\sprod{g_\lambda,y}}^2 =0.
\end{align*} 
The proof of Theorem \ref{theorem:GaborCond} then implies that $\widehat{\cH}_p=0$ for $p =0$, $1$, i.e. that $\cH_0=\cH_1=0$. Remembering that $\cH_p(i)=H_{i,i-p}$, we arrive at
\begin{align*}
	0 = x(i)\bar{x}(i) - y(i)\bar{y}(i) = x(i)\bar{x}(i-1)-y(i)\bar{y}(i-1), \quad i=0, \ldots, N-1.
\end{align*}
The first equality simply says that $\abs{x(i)}=\abs{y(i)},$ i.e. that there exists numbers $\epsilon_i \in \bT$ so that $x(i)=\epsilon_i y(i)$ for all $i$. Inserting this into the second equation yields
\begin{align*}
	0= y(i) \bar{y}(i-1)(\epsilon_i \bar{\epsilon}_{i-1} -1).
\end{align*}
Since all entries of $y$ are assumed to be nonzero, it follows that $\epsilon_i=\epsilon_{i-1}$, i.e. $\epsilon_i=\epsilon_0 =: c\in \bT$ for all $i$. Hence $x=cy$ for a $c \in \bT$, and $x$ and $y$ are equal mod $\bT$.  
\end{proof}

\begin{remark} A similar result was proven in \cite{eldar2015sparse}. There, it was only assumed that $\sprod{g, g_{p,\ell}}\neq 0$ for $p=0$ and all $\ell \in \Z_N$. However, in this case further constraints on the generators need to be made: $g$ must be a window of length $W \geq 2,$ where $N \geq 2W-1$ and $N$ and $W-1$ are coprime. Our result, on the other hand, works for more general generators and any $N.$
\end{remark}
\end{subsection}

\begin{subsection}{Generators which allow phase retrieval}
 We will present two types of signals, one random and one deterministic, which satisfy condition \eqref{eq:GaborCond}, and thus can be used for phase retrieval of signals from all $N^2$  Gabor measurements.

\begin{subsubsection}{Complex random vectors as generators}
We start by considering a probabilistic approach, a common strategy in signal recovery in general.
\begin{proposition}
	Let $g$ be a vector in $\C^N,$  randomly distributed according to the complex standard normal distribution. Then, the condition $\sprod{g,g_\lambda}\neq 0$ for all $\lambda \in \Z_N^2$ is satisfied with probability $1$.
\end{proposition}

\begin{proof}
	Since there are only finitely many $\lambda$'s, it suffices to prove that $\sprod{g,\Pi_\lambda g}\neq 0$ with probability $1$ for one arbitrary $\lambda$. Since $\Pi_\lambda$ is a unitary operator, there exists an orthonormal basis $\left(q_i\right)_{i=1}^N$ of $\C^N$ and $c_i \in \bT$ with
	\begin{align*}
		\Pi_\lambda = \sum_{i=1}^N c_iq_iq_i^*.
	\end{align*}
	If we expand $g$ in this basis, i.e. $g= \sum_{i} h_i q_i$, then the vector $h \in \C^N$ will also be distributed according to  the complex standard normal distribution \cite{gallager08Gauss}. We have $\Pi_\lambda g = \sum_i c_i h_i q_i,$ and hence
	\begin{align*}
	 \sprod{g,\Pi_\lambda g} = \sum_{i =1}^N c_i \abs{h_i}^2.
	\end{align*}
	In order for $g$ to not satisfy \eqref{eq:GaborCond}, the random variable $\mathfrak{h}=\left(\abs{h_i}^2\right)_{i=1}^N$ on $\R^N_+$ must hence lie in the subspace of $\R^N$ defined by
	\begin{align*}
		\set{v : \sum_{i =1}^n c_i v_i = 0}.
	\end{align*}
	Since this space has dimension $N-1,$ the set has Lebesgue measure zero. If we prove that $\mathfrak{h}$ has a distribution which has a density with respect to the Lebesgue measure on $\R_+^N$ which is almost never zero, we are done. This is however not hard to see, since the variables $\abs{h_i}^2= \abs{a_i}^2+\abs{b_i}^2,$ $i=1,\ldots,N$ are independently distributed according to the $\chi^2_2$-distribution, which has density $\rho(x)=\frac{1}{2}\exp(-x/2)$ on $\R_+$. \end{proof}

\end{subsubsection}

\begin{subsubsection}{Difference sets as generators}
The second example are so-called difference sets, a construction coming from combinatorial design theory \cite{dinitz1992contemporary}. The set of all modulations of a characteristic function of difference set
was shown to achieve the Welch bound in \cite{xia2005achieving}. We will show that the set of all modulations and translations of a difference set has the property desired for phase retrieval.

\begin{definition}
	A subset $\cK = \{u_1,\ldots,u_K\}$ of $\Z_N$ is called an $(N,K,\nu)$ \textit{difference set} if the $K(K-1)$ differences
	$$(u_k -u_l) \mod N, \quad k \neq l$$
	take all possible nonzero values $1,2,\ldots,N-1,$ with each value appearing exactly $\nu$ times.
\end{definition}

\begin{example} \label{ex:dfs}
Let $N=7.$ The subset $\cK = \{ 1,2,4\}$ is then a $(7,3,1)$ difference set. We can check this by considering all
possible differences modulo $7$,

$$\begin{tabular}{c|ccc}
 - & 1 & 2 & 4  \\ \hline
 1 & - & 6 & 4  \\
 2 & 1 & - & 5  \\
 4 & 3 & 2 & -  
\end{tabular}$$
and confirming that indeed every value from $1$ to $6$ appears exactly one time.
\end{example}
Given  a difference set  $\cK$  with parameters $(N,K,\nu)$ we denote by $\chi_{\cK} \in \{0,1\}^N$ its
characteristic function:
$$ \chi_{\cK}(j) = \begin{cases}
      \hfill 1,    \hfill & \text{if } j \in \cK \\
      \hfill 0,   \hfill & \text{if } j \notin \cK.\\
      \end{cases}
$$
We now prove that if such characteristic functions are used as generators, the corresponding Gabor frames will satisfy \eqref{eq:GaborCond}, and hence allow phase retrieval for arbitrary signals.
\begin{proposition} \label{th:diffsets}
Let $N$ be an integer with a prime factorization $N=p_1^{a_1}\ldots p_r^{a_r}.$ Let $\cK$ be a difference set with parameters $(N,K,\nu),$ such that
\begin{equation} \label{eq:condl}
\nu, K < \min\{ p_1,\ldots, p_r\}.
\end{equation}
Then, for $g=\chi_\cK,$
\begin{equation}  \label{eq:ggl}
\langle g, g_\mu \rangle \neq 0 \text{ for every } \mu \in \Z_N^2.
\end{equation}
\end{proposition} 
\begin{proof}
 Let $\mu =(q,j),$ with both $q,j\neq 0.$  By just using the definition of $g_\mu$ and $\cK$  we obtain
 \begin{equation} 
  \langle g, g_\mu \rangle  = \sum_{n \in \Z_N} g(n)  (M_j T_q g )(n) = \sum_{n \in \Z_N} g(n)g(n-q)
 \omega^{jn} = \sum_{\substack{n \in \cK \text{ and } \\ n-q \in \cK}} \omega^{jn}.
 \end{equation}
Now, taking into account the nature of a difference set, we can conclude that in the set
$$\{ n: n \in \cK, n-q  \in \cK \}$$
there will be always exactly $\nu$ elements (because for $q \in \Z_N$ there are exactly $\nu$ ways to be written as a difference of elements in $\cK,$ and $n-(n-q)$ are such differences). 

If $\nu=1,$ we are left with a single $\omega^{j n_0}$ and then certainly the sum is different from zero.

If $\nu\neq 1,$ we have a sum of $\nu$ different $N$-th roots of unity,  and we will show that with the given assumptions on the 
difference set, \eqref{eq:ggl} holds. 
We use the following result from \cite{lam2000vanishing} about the vanishing sums of roots of unity.
The main theorem in this article states that for any $N=p_1^{a_1}\ldots p_r^{a_r},$ the only possible amounts of $N$-th roots of unity that can sum up to zero is given by $M_1 p_1 + \ldots + M_r p_r.$ Here the $M_i$ are any non-negative
integers ($0$ is included). Now it is clear that the condition $\nu < \min \{p_1,\ldots p_r\}$ will ensure that we will never have a vanishing sum. 

If now $\mu = (0,j)$ the sum will go over the full set $\cK,$ and since $K < \min \{p_1,\ldots p_r\},$ again this sum is non vanishing.

Finally, in the last case $\mu = (q, 0),$ we have a sum of $\nu$ ones, and therefore we have proven \eqref{eq:ggl} for all cases $\mu \in \Z_N^2.$
\end{proof}

\begin{example}
 We now provide some examples of families of difference sets, which satisfy the condition from Proposition~\ref{th:diffsets}.
 
\textit{Family 1:}  Quadratic Difference Sets. Let $q=p^r = 3\, (\m 4)$ be a power of a prime and 
$$N=q, \, K=\frac{q-1}{2}, \, \nu= \frac{q-3}{4}.$$
Then $u = \{ t^2: t \in \Z_N\backslash \{ 0 \}\}$ is a $(N,K,\nu)$ difference set. If $r=1,$ condition \eqref{eq:condl} is satisfied.

\textit{Family 2:}  Quartic Difference Sets. Let $p=4a^2+1$ be a prime with $a$ odd, and 
$$ N=p,\, K= \frac{p-1}{4}, \, \nu = \frac{p-5}{16}.$$
Then $u= \{ t^4: t \in \Z_N\backslash \{ 0 \}\}$ is a $(N,K,\nu)$ difference set and additionally $K,\nu < N.$ 

Many other examples can be found in the paper \cite{xia2005achieving}, or in the La Jolla Difference Set Repository at \url{http://www.ccrwest.org/ds.html}.
\end{example}
\end{subsubsection}

\end{subsection}

\begin{subsection}{Generators which do not allow phase retrieval} 
We now consider two cases for which condition \eqref{eq:GaborCond} is \emph{not} satisfied, and show that this in fact implies that the Gabor frames do not allow phase retrieval in these cases.
\label{subseq:negativeExamples} 

\begin{proposition} \label{prop:negex}
Let $g \in \C^N$ be a generator such that one of the following two conditions is satisfied
\begin{align} 
\sprod{g,g_{\hat{p},\ell}} & = 0, \quad \text{ for fixed } \hat{p} \in \Z_N \backslash \{0 \} \text{ and all } \ell \in \Z_N. \label{eq:condg1} \\
\sprod{g, g_{\hat{p},\ell}}& = 0, \quad \text{  for } \hat{p}=0 \text{ and all } \ell \in \Z_N \backslash \{0 \}. \label{eq:condg2}
\end{align}
Then, the corresponding Gabor frame $G = \left( g_\lambda \right)_{\lambda \in \Z_N^2}$ does not allow phase retrieval for $\C_N.$
\end{proposition}
\begin{proof}
Let us first assume that condition \eqref{eq:condg1} is satisfied. We consider the matrix $H_1 \in \bH^{N \times N}$, defined by
\begin{align*}
	H_1 = e_0 e_{-\hat{p}}^* + e_{-\hat{p}}e_0^*
\end{align*}  
($e_0$ is the 'first unit vector' - remember that we are always considering indices from $\Z_N$). This matrix has rank 2, and it lies in the kernel of the PhaseLift operator associated to the Gabor frame defined in \eqref{eq:opA}. To see this, note that using the notation of the proof of Theorem \ref{theorem:GaborCond} we have
\begin{align*}
	\cH_p(i)= H_{i,i-p} = \begin{cases}
	1 &\text{ if } i=0, p= \hat{p}, \\
	1 &\text{ if } i={-\hat{p}}, p= -\hat{p}, \\
	0 & \text{ else.}
	\end{cases}
\end{align*}
In other words, $\cH_p =0$ for all $p \neq \pm \hat{p}$. Since $\sprod{g,g_{p, \ell}} = \omega^{-\ell p}\overline{\sprod{g,g_{-p, -\ell}}}$, equation \eqref{eq:condg1} also implies $\sprod{g, g_{-\hat{p},\ell}}=0$ for all $\ell \in \Z_N$. These two facts prove that  
$$\widehat{\cH}_p(\ell)\sprod{g,g_{p, \ell}}=0$$ 
for all $\ell$ and $p$. Using the technique of the proof of Theorem \ref{theorem:GaborCond} backwards, it follows $\cA(H_1)=0$. The matrix $H_1$ that we have found has rank $2$ and it is in the kernel of $\cA.$ Therefore, by Theorem \ref{th:PRfull}, the Gabor frame can not allow phase retrieval. 

Now we assume that \eqref{eq:condg2} is satisfied. In this case we define a rank 2 matrix in $\bH^{N \times N}$ by \begin{align*}
	H_{2}= e_0e_0^* - e_1e_1^*.
\end{align*}
For this matrix, $\cH_p=0$ for $p \neq 0$. Also $\widehat{\cH}_{0}(0)= \sum_i H_{i,i}=0$. Because of these two facts and the assumption on $g$, we again have
$$\widehat{\cH}_p(\ell)\sprod{g,g_{p, \ell}}=0 \quad \text{ for all} \ (p, \ell) \in \Z_N^2,$$ 
and $H_2$ will by the same argument as before be in the kernel of $\cA$. Phase retrieval is again not possible.
\end{proof}
\begin{figure}
\begin{align*}
	H_1 = \begin{bmatrix}
	  0&  &   1 &  &  \\
	  & \ddots  &   &  &   \\
	1 &  &  \ddots   &  &   \\ 
	  &  &    &  \ddots &   \\
	 &  &     &  &   0
	  	\end{bmatrix} \quad 	H_2 = \begin{bmatrix}
	  1&  &    &  &  \\
	  & -1 &   &  &   \\
	 &  &  0  &  &   \\ 
	  &  &    &  \ddots &   \\
	 &  &     &  &   0
	  	\end{bmatrix}
\end{align*} 
\caption{The matrices $H_1$ and $H_2$ used in the proof of Proposition \ref{prop:negex}}
\end{figure}

\begin{example}
We now give two examples, for which the conditions of the previous proposition are satisfied. 

\textit{ Short windows:} The condition \eqref{eq:condg1} is satisfied if the generator $g$ is a ``short window''. More precisely, if $\supp g \sse [K_1,K_2]$ for $\abs{K_1-K_2}< \frac{N}{2},$ then $g$ and $M_\ell T_pg$ will have disjoint supports for some $p$'s and hence have a vanishing scalar product. Using a window as a generator is a core idea in short-time Fourier analysis~\cite{pfander2013gabor}.

\textit{Alltop sequence:} It can be easily shown that the much celebrated Alltop sequence \cite{alltop1980complex}, defined as 
$\big( \frac{1}{\sqrt{N}} \omega^{n^3} \big)_{n=0}^{N-1},$ has the property \eqref{eq:condg2}. This generator is often and successfully used in sparse signal recovery from linear Gabor measurements \cite{bajwa2010gabor, pfander2013gabor}.

However, both these families of signals can not be used for phase retrieval, when we are interested in recovery of all signals in $\C^N.$ 
\end{example}

\end{subsection}

\end{section}

\begin{section}{An injectivity condition for sparse signals} \label{sec:sparse}
\begin{subsection}{Sparse phase retrieval via the PhaseLift operator}
We will now consider signals which are sparse in a dictionary. A dictionary $D$ is a set of $d$ vectors in $\C^N$, and it is identified with the matrix formed when writing the $d$ vectors as its columns. The class of signals which are $k$-sparse in the dictionary $D$, or simply $kD$-sparse signals, is
\begin{align*}
	\mathcal{C}_{k,D} = \set{x \in \C^N \vert \ \exists \ z \in \C^d,\norm{z}_0\leq k, \text{ s.t. } x=Dz},
\end{align*}
where $\norm{z}_0$ denotes the number of non-zero coefficients in $z$. If $D=I$, we will  omit the dictionary and simply speak of $k$-sparse vectors $\mathcal{C}_k$. We will also, to increase readability, speak of $kD$-phase retrieval instead of phase retrieval for $\mathcal{C}_{k,D}.$

Since sparse vectors in some sense are $k$-- and not $N$-dimensional, one would hope that the number of measurements required to recover them is smaller (for us, $\Lambda$ should contain less elements.) This, and other, questions were considered and answered for general measurement vectors in \cite{ohlsson2014conditions} and \cite{wang2014phase}. 
A counterpart of Theorem \ref{th:PRfull} in the sparse setting has up to know not been stated and proved.  We will prove now an injectivity condition for sparse signals, and then use it for the case of Gabor measurements as in the previous section.

 For a given dictionary $D=\left( d_i \right)_{i=1}^d$ and a set of indices $\cK \sse \{ 1, \dots d \},$ 
we denote $W_\cK =\Span \{ d_i \}_{i \in \cK}.$
With the help of this notion, we can characterize sets or measurement vectors which allow $kD$-phase retrieval. Let $\cA$ be the PhaseLift operator defined in \eqref{eq:opA}.

\begin{theorem} \label{th:PhasLiftSparse}
	Given the notations from above, the following two statements hold.
	\begin{enumerate}
		\item If for every $\cK$ with $\abs{\cK}=2k$, the kernel of $\mathcal{A}$ does not contain rank $1$ or $2$ Hermitian matrices whose range is in $W_\cK,$  then the vectors $\left(f_i\right)_{i=1}^m$ allow $kD$-phase retrieval.
		\item If $\left(f_i\right)_{i =1}^m$ is allowing $kD$-phase retrieval,  then for every $\cK$ with $\abs{\cK}=k$, the kernel of $\mathcal{A}$ does not contain rank $1$ or $2$ Hermitian matrices with range in $W_\cK$.
	\end{enumerate}
\end{theorem}

\begin{proof}
	Let us start by proving $(1)$ by contraposition. Assume that $\left(f_i\right)_{i =1}^m$ is not allowing $kD$-phase retrieval. Then there exists $x \neq y \text{ mod } \mathbb{T}$, both $kD$-sparse, for which 
$$\sprod{ x x^{*},f_if_i^{*}}_{HS}=\abs{\sprod{f_i, x}}^2=\abs{\sprod{f_i, y}}^2=\sprod{ y y^{*}, f_if_i^{*}}_{HS},$$
 i.e. $xx^{*}-yy^{*}$ is a Hermitian matrix in the kernel of $\mathcal{A}$. If the sparse representations of $x$ and $y$ are given by $x= Dz_x$ and $y=Dz_y$, we see that $\text{ran}(xx^{*}-yy^{*}) \sse (W_{\supp z_x \cup \supp z_y})$ and further it has rank less than or equal to $2$. If we knew that the rank is at least one, $\abs{\supp z_x \cup \supp z_y}\leq 2k$ would imply the claim.
	
	To see that the condition $x \neq y \text{ mod } \mathbb{T}$ in fact implies this, assume, towards a contradiction, that this is not the case, i.e that $xx^*-yy^*=0.$ Since $x \neq y \text{ mod } \mathbb{T},$ both vectors are non-zero. Hence there exists a vector $v \in \C^N$ such that $\sprod{x,v}\neq 0.$ Multiplying $0=xx^*-yy^*$ with this $v$ and rearranging terms, we arrive at 
$$x=\frac{\sprod{y,v}}{\sprod{x,v}}y,$$
 i.e $x= \lambda y$ for a $\lambda \in \C.$ Again plugging this into  $0=xx^*-yy^*$ yields $\abs{\lambda}=1$. This is a contradiction. 
	
	Let us now turn to $(2)$. Suppose that there exists a $\cK$ such that the kernel $\mathcal{A}$ contains a Hermitian matrix $H$ with rank $1$ or $2$ with range in $W_\cK$. By the spectral theorem, there exists an orthonormal basis $(\varphi_j)$ of $\C^N$ consisting of eigenvectors of $H$, corresponding to real eigenvalues $(\lambda_j)$. It is clear that $H$ may be written as $\sum_j \lambda_j \varphi_j \varphi_j^{*}$.
	
	Because of the bounded rank and the fact that $H\neq 0$, either one or two of the eigenvalues are non-zero. It is clear that the eigenvectors corresponding to those eigenvalues  are vectors in $W_\cK$, since they form a basis of the range of $H$. Thus, they are $kD$-sparse. 
	
	Let us first consider the case where only one eigenvalue is different from zero. If we write $x= \sqrt{\abs{\lambda_1}}\varphi_1,$ then we have $H= \pm xx^{*}$ and hence
	\begin{align*}
		0 = \sprod{x x^{*}, f_if_i^{*}}= \abs{\sprod{f_i,x}}^2=0.
	\end{align*}
	This means that the two $kD$-sparse vectors $x$ and $0$ have the same phaseless measurements, although $x \neq 0 \mod \mathbb{T}$.
	
	The other case is dealt with similarly: here we write $x= \sqrt{\abs{\lambda_1}}\varphi_1$, $y = \sqrt{\abs{\lambda_2}}\varphi_2$ and conclude that $H = \pm xx^{*} \pm yy^{*}$, where the signs depend on the signs of the eigenvalues. If the signs are equal, we see that $\abs{\sprod{f_i,x}}^2 + \abs{\sprod{f_i,y}}^2 =0$ and we have again found $kD$-sparse vectors which are measured $0$. If the signs are not equal, we see that $\abs{\sprod{f_i,x}}^2 - \abs{\sprod{f_i,y}}^2 =0,$ and hence $x$ and $y$ are measured equally. They are $kD$-sparse and cannot be equal $\text{mod } \mathbb{T}$, since they are orthogonal.
\end{proof}

\end{subsection}

\begin{subsection}{Signals sparse in the standard basis}
Let us start by considering vectors which are sparse in the standard sense, i.e. $D=I$. We will prove a condition under which a subset of our Gabor frame $\set{g_\lambda\ , \ \lambda \in \Z_N^2}$ with $\sim k^3$ elements allows $k$-sparse phase retrieval, when $N$ is prime. For general $N$, it would still be possible to go below the full set of measurements, $N^2.$ We will need a special form of the discrete uncertainty principle, which involves the sum of the ``spread'' of the signal and its Fourier transform. Let us start with a general observation.
\begin{lemma} \label{lemma:uncertaintyPrinciple}
	Assume that for all non-zero vectors $f \in \C^N$ 
	\begin{align} \label{eq:uncertainty}
		\norm{f}_0 + \Vert \hat{f}\Vert_0 \geq N - \theta_N
	\end{align}
	holds for some number $\theta_N.$ Then, if $f$ is $k$-sparse ($\norm{f}_0=k$), and
	$\hat{f}$ has not less than $\theta_N+k+1$ zero-entries, then $f$ necessarily has to vanish.
\end{lemma}
This statement follows immediately by contradiction. The question is whether \eqref{eq:uncertainty} is a reasonable assumption. In \cite{tao03uncertainty} it is proved that when $N$ is prime, \eqref{eq:uncertainty} holds with $\theta_N$ equal to $-1.$ For general $N,$ by the standard multiplicative uncertainty principle and the geometric mean-arithmetic mean inequality, one can derive \eqref{eq:uncertainty} with $\theta_N = N-2\sqrt{N}.$ A more involved inequality for general $N$ was obtained in \cite{meshulam2006uncertainty} and will be discussed later on.

Before we proceed with a condition on the generator $g$ for sparse phase retrieval, we will first prove a more general statement about recovery of sparse matrices from \textit{linear} measurements, which is interesting on its own. We will be interested in the following class of signals,
\begin{align*}
	\mathfrak{H}_K = \set{ H \in \C^{N \times N} : \exists \cK \sse [1, \dots N], \abs{\cK}=K :  H_{ij}= 0 \text{ if } 
	(i,j) \notin \cK \times \cK}.
\end{align*}

\begin{theorem} \label{th:GaborKsparseOperatorVersion}
	Let $N$ be such that the uncertainty principle \eqref{eq:uncertainty} holds, and let $\lambda=(p,l) \in \Z_N^2.$ Let $g$ have the following property: for each $\ell$, the sequence $c_p=(\sprod{g,g_{p,\cdot}})$ formed by letting $\ell$ run obeys
	\begin{align} \label{eq:GaborKsparseCond}
		\theta_N+K+1 \leq \norm{c_p}_0 \leq \hat{k}
	\end{align}
	for some $K$ and $\hat{k}.$ Then, for any subsets $A \sse \Z_N$, $B \sse \Z_N$ with 
	\begin{align*}
	 \abs{A} \geq \theta_N+\hat{k}+1, \quad	 \abs{B} \geq \theta_N + K^2-K+2,
	 \end{align*}
	the following holds. If a matrix $H \in \mathfrak{H}_K$ satisfies
	$(\sprod{g_\lambda g_\lambda^*, H}_{HS})_{\lambda \in A \times B} = 0,$ then $H=0.$
\end{theorem} 

\begin{proof}
Let $H \in\mathfrak{H}_K$ satisfy $\sprod{g_\lambda g_\lambda^*, H}_{HS}=0$ for $\lambda \in A \times B$ and let $\cK$ be such that  $H_{ij}= 0$ if $(i,j) \notin \cK \times \cK$. We will prove that $H$ then must be $0.$ Recall the notation from the proof in Theorem \ref{th:PRfull},
$\cH_p(i) = H_{i,i-p}.$ Since $H_{i,i-p}$ is zero, if $(i,i-p)$ is not in $\cK \times \cK,$ we can conclude that
	\begin{align*}
		\cH_p(i) = H_{i,i-p} = 0 \quad \text{ if } i \notin \cK \cap (\cK + p).
	\end{align*}
	This  proves the following properties: 
	\begin{enumerate}
		\item \label{HlkSparse} The vectors $\cH_p$ are $K$-sparse.
		\item \label{HlZeroForManyL} $\cH_p$ is zero for all but at most $K^2-K+1$ different values for $p$. To see this, notice first that $\cH_p =0$ if $p \notin \cK - \cK$. This is because if $\cH_p(i)\neq 0$, then $i \in \cK$ and there additionally exists a $j\in \cK$ with $i=j+p$. It follows $p=i-j \in \cK-\cK$. And we know that the set $\cK-\cK$ has at most $\abs{\cK}(\abs{\cK}-1) +1 = K^2-K+1$ elements. 
	\end{enumerate} 
	
	Now using the same argument as in the proof of Theorem \ref{theorem:GaborCond}, we arrive at
	\begin{align} \label{eq:FourTransAtEta}
		0 = N\sprod{g_\lambda g_\lambda^*, H} = \sum_{p, \ell}\omega^{-jp}\omega^{\ell q} \widehat{\cH}_p(\ell)\sprod{g,g_{p, \ell}} \quad \text{ for all} \ \lambda = (q, j) \in A \times B .
	\end{align}
	Fixing $j$, the sum in  \eqref{eq:FourTransAtEta} is the value at $j$ of the discrete Fourier transform of the vector $V^q$ defined as
	\begin{align} \label{eq:Vq}
		V^q(p)= \sum_{\ell} \omega^{\ell q}\widehat{\cH}_p(\ell)\sprod{g,g_{p, \ell}}.
	\end{align}
	Because of (\ref{HlZeroForManyL}), these vectors are all $(K^2-K+1)$-sparse. Further, \eqref{eq:FourTransAtEta} proves that their Fourier transforms vanish at all $j\in B,$ i.e. at $\theta_N+(K^2-K+2)$ points. The discrete uncertainty principle \eqref{eq:uncertainty} implies that $V^q$ must equal zero.
	
	Considering \eqref{eq:Vq}, the fact that $V^q(p)=0$ proves that the inverse Fourier transform of the vector, which we denote by
	\begin{align*}
		w^p(\ell)=\widehat{\cH}_p(\ell)\sprod{g, g_{p,\ell}}
	\end{align*}	 
	vanishes at the values $q\in A$, i.e. at $\theta_N+\hat{k}+1$ values. Because of our assumption on $g$, $w^p$ is however $\hat{k}$-sparse. We can therefore again conclude that
	\begin{align*}
		\widehat{\cH}_p(\ell) \sprod{g, g_{p,\ell}} = 0 \quad \text{ for all} \ (p, \ell)\in \Z_N^2.
	\end{align*}
	Hence, if $\sprod{g, g_{p,\ell}} \neq 0$, $\widehat{\cH}_p(\ell)$ must be $0$. Due to our assumption on $g$, this happens for at least $\theta_N+2k+1$ $\ell$'s for every $p$. Because of \ref{HlkSparse}, this is sufficient to prove that $\cH_p =0$ for all $p$, and $H$ therefore must be~$0$.
\end{proof}

We now use the theorem we have just proved, to provide a condition, when a Gabor frame can do $k$-sparse phase retrieval. 
\begin{theorem} \label{theorem:GaborKsparseCond}
Let $N$ be such that the uncertainty principle \eqref{eq:uncertainty} holds, and let $\lambda=(p,l) \in \Z_N^2.$ Let $g \in \C^N$ be a generator which satisfies the following condition: for each $\ell$, the sequence $c_p=(\sprod{g,g_{p,\cdot}})$ formed by letting $\ell$ run obeys
	\begin{align} \label{eq:GaborKsparseCond}
		\theta_N+K+1 \leq \norm{c_p}_0 \leq \hat{k}
	\end{align}
	for some $K=2k$ and  some $\hat{k}.$ Then, for any subsets $A \sse \Z_N$, $B \sse \Z_N$ with 
	\begin{align*}
	 \abs{A} \geq \theta_N+\hat{k}+1, \quad	 \abs{B} \geq \theta_N + (2k)^2-2k+2,
	 \end{align*}
	 the set
\begin{align} \label{eq:GaborSubSystem}
		\set{g_\lambda, \,  \lambda \in A \times B}
	\end{align}
	allows $k$-sparse phase retrieval.
\end{theorem}
\begin{proof}
We will use part $(1)$ of Theorem \ref{th:PhasLiftSparse} for $D=I,$ to show that $k$-sparse phase retrieval is possible for the system \eqref{eq:GaborSubSystem}. Let $H$ be an Hermitian operator with values in $\C^N_\cK=\{x \in \C^N, \, \supp(x) \subseteq \cK\}$ for some $\cK$ with $\abs{\cK}=2k$ for which $\cA(H)=0$ (where $\cA$ is the PhaseLift operator associated with \eqref{eq:GaborSubSystem}). We will prove that $H$ must be zero, from which the claim follows. Since the range of $H$ is contained in $\C^N_\cK$, we know that $H_{i,j}=0$ if $i \notin \cK$. Since $H_{i,j} = \overline{H_{j,i}}$, we also have $H_{i,j}=0$ if $j \notin \cK$. We can conclude that $H \in \mathfrak{H}_{2K}$, and by Theorem \ref{th:GaborKsparseOperatorVersion} it immediately follows that $H$ is zero, thus the theorem is proved.
\end{proof}

\begin{remark} 
We would like to state the following remarks related to Theorems \ref{theorem:GaborKsparseCond} and \ref{th:GaborKsparseOperatorVersion}.
\begin{enumerate}
 \item If $\theta_N+4k^2-2k+2 \geq N$, then the same theorem holds for $B =\Z_N$ and any $A$ with $\abs{A}\geq \theta_N+\hat{k}+1$. We can therefore also in this case reduce the number of measurements from $N^2$ to $(\theta_N+\hat{k}+1)N.$ Also note that since $\hat{k}$ must not be smaller than $2k$, the theorem does not yield any enhanced results for non-sparse vectors (we need the sequences $c_p$ to be $\hat{k}$-sparse to reduce the number of measurements, but to have at least $2k$ nonzero elements to ensure injectivity for $k$-sparse signals).

\item We note, that when $N$ is prime, the conditions of Theorem \ref{theorem:GaborKsparseCond} become much simpler. Namely, $$2k \leq \norm{c_p}_0 \leq \hat{k},$$ and the sets $A$ and $B$ should fulfill
$$\vert A \vert \geq \hat{k}, \quad \vert B \vert \geq (2k)^2-2k+1.$$ Thus, for example, if
we can find a Gabor system for which the inequality \eqref{eq:GaborKsparseCond}   is fulfilled as an equality, we will be able to do $k$-sparse phase retrieval with order of only $O(k\min(N,k^2))$ measurements. 

\item When $N$ is not prime, we have $\theta_N = N-2\sqrt{N},$ and the number of needed measurements is not as good as in the prime case, since we obtain
$$\vert A \vert \cdot \vert B \vert \geq (N-2\sqrt{N}+\hat{k}+1)(N-2\sqrt{N}+(2k)^2-2k+2),$$
but some improvement over $N^2$ could still be obtained in some cases. Furthermore, an extension of \cite{tao03uncertainty}  from $N$ prime to general $N$ was published in \cite{meshulam2006uncertainty}, in the form of the following property:

Let $d_1 <d_2$ be two consecutive divisors of $N.$ If $d_1 \leq k = \norm{f}_0 \leq d_2,$ then
$$\Vert \hat{f} \Vert_0 \geq \frac{N}{d_1d_2} (d_1+d_2-k).$$
Our function $\theta$ will in this case explicitly depend on $k$ and be equal to $N-k+\frac{N}{d_1d_2} (d_1+d_2-k).$ The smaller this value is, the less measurements will be needed for $k$-sparse injectivity. 

\item Theorem \ref{th:GaborKsparseOperatorVersion} is interesting from a different perspective, since $\mathfrak{H}_K$ can be viewed as a set of $K^2$-sparse vectors in $\C^{N^2}$ whose sparsity has a special structure. Thus, we have provided a deterministic construction which can theoretically recover those vectors from $O(K^3)$ linear measurements.
This is interesting since we know from conventional compressed sensing results \cite{foucart2013amathematical}, that deterministic constructions for stable recovery of $K^2$-sparse vectors require $O(K^4)$ linear measurements, whereas random constructions only need $O(K^2)$ measurements. Finding deterministic constructions which can accept sparsity levels on the order higher than the square root of the number of measurements is known as breaking the ``square-root bottleneck'' \cite{mixon2014explicit}.

Although in our case, $O(m)$ measurements are needed for sparsity level $m^{2/3},$ one has to bear in mind that the sparsity of the vectors is structured, and that our result is only about the injectivity of the measurements. 
In particular, we do not prove any recovery guarantees for a specific algorithm. Hence, we have not broken the square-root bottleneck, but the theorem can be seen as a step towards providing new results in this direction.

\end{enumerate}

\end{remark}
\end{subsection}

\begin{subsection}{Functions window in the Fourier domain as generators.}
As in the previous section, we now provide an example of a generator $g$ which fulfills our condition.
\begin{proposition} \label{prop:Fourierwindow}
	Let $N$ be prime and $2k+1< N$. 
	Further, let $v \in \C^N$ be a window of length $k+1,$ $v= \chi_{[0,k]}$, where $\chi_A$ denotes the characteristic function on the set $A \sse \Z_N.$ 
	Moreover, let $g$ be defined by $\hat{g}=v$, Then, $g$ satisfies \eqref{eq:GaborKsparseCond} with $\hat{k}=2k+1$ and therefore, $(2k+1)\min(4k^2-2k+1,N)$ measurements from the Gabor frame with $g$ as generator will do $k$-sparse phase retrieval.
\end{proposition}
\begin{proof}
	 The Plancherel formula implies that 
	 \begin{align*}
		c_p(\ell)=\sprod{\hat{g}, T_\ell M_p\hat{g}} \quad \text{for all }  \, (p,\ell) \in \Z_N^2
	\end{align*}
	Therefore,
	\begin{align} \label{eq:windowCond}
		\sprod{\hat{g}, T_\ell M_p\hat{g}} = \sum_{m\in \Z_N} \overline{v(m)}v(m-\ell)\omega^{p(m-\ell)} = \sum_{\substack{m \in [0,k] \text{ and } \\ m-\ell \in [0,k]}}\omega^{p(m-\ell)},
	\end{align}
	because of the way we defined $v$. Note that since $2k<N$, this sum is empty for $\abs{\ell}>k$.  Therefore, the sequences $c_p$ are $2k+1$-sparse for every $p,$ i.e. $\hat{k}$-sparse. 
		
	It remains to prove that for $\abs{\ell} \leq k,$ the expression above is not zero, and hence $\norm{c_p}_0=2k+1.$
	It suffices to consider $\ell\geq 0$, since the other case can be obtained from this one by the substitution $\ell \to - \ell$. Using the formula for geometric sums, we obtain
	\begin{align*}
	 \sum_{\ell \leq m \leq k }\omega^{p(m-\ell)} = \sum_{n=0}^{k-\ell}\omega^{pn}= \begin{cases} \dfrac{1-\omega^{p(k-\ell+1)}}{1-\omega^{p}},\ & \text{ if } p \neq 0, \\
	 k-\ell+1, \ & \text{ if }  p=0. \end{cases}
	\end{align*}
	The only way this could be zero when $\ell \leq k$ is that $p \neq 0$ and  $1-\omega^{p(k-\ell+1)}=0$. This would however mean that $N$ is a divisor of $p(k-\ell+1)\neq 0$. Since $N$ is prime and both $p$ and $(k-\ell+1)$ are smaller than $N$, this cannot be the case. Therefore, from Theorem \ref{theorem:GaborKsparseCond} we conclude that any subsets $A \subseteq \Z_N,$ $B \subseteq \Z_N$ with 
	$$\vert A \vert \geq 2k+1, \quad \vert B \vert \geq (2k)^2-2k+2$$ will do $k$-sparse phase retrieval.
\end{proof}

\begin{remark}
	The choice of $\hat{g}$ as a characteristic function of $[0,k]$ is not necessary -- any generically chosen function with support on $[0,k]$ will also lead to a Gabor system with the same properties.
	To see this, note that if $\abs{\ell}<k$, the expression \eqref{eq:windowCond} is a non-trivial polynomial in the variables $\mathfrak{re}(v), \mathfrak{im}(v)$. Since we have proved that there is a particular choice of $v$ so that all polynomials do not vanish on $v$, \eqref{eq:GaborKsparseCond} will be satisfied for generic $v$.
\end{remark}

\begin{remark}
		The matrices provided to prove that the frames considered in Section \ref{subseq:negativeExamples} do not allow phase retrieval were all matrices with range in $\C^N_\cK$ for a $\cK$ with $\abs{\cK}=2$. Hence, the considerations made there in fact proved that the frames are not allowing phase retrieval for $\mathcal{C}_k$ for any $k\geq 2$ (although they might still allow phase retrieval for some other class of signals).
	\end{remark}
	 \end{subsection}

\begin{subsection}{Signals sparse in Fourier domain}
	 After spending some time discussing the standard sparsity case, it is worth noting that similar results hold for signals which are sparse in the Fourier basis (dictionary) $F$. Recall the famous commutation relation of $F$ with translations and modulations:
	 \begin{align} \label{eq:commModTransFor}
	 	\Pi_{(p,\ell)}F=M_\ell T_pF =FT_{\ell}M_{-p}= \omega^{\ell p}F\Pi_{(-p,\ell)}\quad \text{ for all } (p,\ell) \in \Z_N^2.
	 \end{align}
	 This formula allows us to translate the results provided in the previous section to this new setting.

\begin{theorem} \label{theorem:GaborFourierKsparseCond}
	Let $N$ be such that the uncertainty principle \eqref{eq:uncertainty} holds and $F$ denote the Fourier basis. Let $g$ have the following property: for each $\ell$, the sequence $\tilde{c}_\ell=(\sprod{g,g_{\cdot,\ell}})$ formed by letting $p$ run obeys
	\begin{align} \label{eq:GaborFourierKsparseCond}
		\theta_N+2k+1 \leq \norm{\tilde{c}_\ell}_0 \leq \hat{k}
	\end{align}
	for some $k$ and $\hat{k}.$ Then, for any subsets $A, B \sse \dZ_N$ with 
	\begin{align*}
	 \abs{A} \geq \theta_N + (2k)^2-(2k)+2, \quad	 \abs{B} \geq \theta_N+\hat{k}+1,
	 \end{align*}
	the set $ \set{g_\lambda, \,  \lambda \in A \times B} $
	allows $Fk$-sparse phase retrieval.
\end{theorem} 

\begin{proof}
We would like to apply Theorem \ref{th:PhasLiftSparse}, with $D=F.$ Using the notation of that theorem, let $H$ be an arbitrary Hermitian matrix with range contained in $W_\cK$ for some $\cK$ with $\abs{\cK}=2k$. We may write $H = FH^FF^*$ for some other Hermitian $H^F$, which then has a range which is contained in $\C^N_\cK$. Let us now proceed as in the proof of Theorem \ref{theorem:GaborCond} and calculate 
	\begin{align*}
		\sprod{\Pi_{(p,\ell)}, H}_{HS} &= \tr(\Pi_{(p,\ell)}^*FH^FF^*)= \tr(F^*\Pi_{(p,\ell)}^*FH^F)= \sprod{F^*\Pi_{(p,\ell)}F,H^F}_{HS}\\
		&=\sprod{\omega^{\ell p}\Pi_{-p,\ell},H^F}_{HS}= \omega^{-\ell p}\widehat{\cH}^F_{-\ell}(p).
	\end{align*}  
	We used the commutation relation \eqref{eq:commModTransFor}, and the fact that $F$ is unitary. We arrive at
	\begin{align*}
	 N\sprod{g_\lambda, H g_\lambda} = \sum_{p, \ell} \omega^{-\ell p} \widehat{\cH}^F_{-\ell}(p)\sprod{g_\lambda,\Pi_{(p, \ell)}g_\lambda}.
	\end{align*}
	This formula is very similar to \eqref{eq:FourTransAtEta}, essentially, the only difference is that the roles of $p$ and $\ell$ have interchanged. Further, the vectors $\cH^F$ have the same properties as the vectors $\cH$ in the proof of Theorem  \ref{theorem:GaborFourierKsparseCond} (since $H^F$ has the same properties as $H$). These two facts makes it clear that we can use the exact same technique as in that proof to prove this theorem. We leave the details to the reader.	
\end{proof}
	 
It is not hard to construct a concrete example of a generator $g$ which fulfills the condition \eqref{eq:GaborFourierKsparseCond}. We only have to note that the roles of translations and modulations have been interchanged. Hence, we should no longer use a $g$ which has short support in Fourier domain, but instead one with short support in spatial domain. With this insight, we may use the exact same steps as in the proof of Proposition \ref{prop:Fourierwindow} to deduce the following.
	 
\begin{proposition}
	Let $N$ be prime and $2k+1< N$. 
	Then $(2k+1)\min(4k^2-2k+1,N)$ measurements from a Gabor frame generated by generic windows $g$ of length $k+1$ will do $Fk$-sparse phase retrieval.
\end{proposition}

\end{subsection}  
  
  \end{section}

\begin{section}{An algorithm for phase retrieval using Gabor measurements} \label{sec:algorithm}

	The idea of the proof of Theorem \ref{theorem:GaborKsparseCond} can be used to design an algorithm to reconstruct signals from their Gabor intensity measurements. We start by recovering $H$, as in the proof, and then we compute the closest rank $1$ operator $xx^*$ by spectral decomposition of $H$. A detailed description is given in Algorithm 1.
	
	\begin{algorithm} \label{alg:algorithm1}
		\caption{Simple Gabor Phase Retrieval (SGPR)}
		\KwData{A generator $g\in \C^N$, sets $A,B \sse \Z_N$,  the measurements   $b(q,j)=N \abs{\sprod{x,g_{q,j}}}^2, (q,j) \in A \times B.$}
		\KwResult{An estimate $x_0 \in \C^N$ of $x$.}
		
		\nl \For{$q=0 \dots N-1$}{
			Solve  $\hat{V}^q(j)=b(q,j), \, j \in B$  for $V^q$.	
		}
		
	\nl 	\For{$p=0 \dots N-1$}{
			Solve  $N \cdot \check{w}^p(q)=	V^q(p), \, q \in A$ for $w^p$. 
		}

		\nl \For{$p=0 \dots N-1$}{
				\For{$\ell=0 \dots N-1$}{
						\If{$\sprod{g,g_{p,l}}\neq 0$}{			
					Set $\hat{\cH}_p(\ell)=w^p(\ell)/\sprod{g,g_{p,l}}$.  
					
					Add $\ell$ to the set $\Lambda_p$.}
		}
			Solve $\widehat{\cH}_p(\ell)=w^p(\ell)/\sprod{g,g_{p,l}}, \, \ell \in \Lambda_p$ for $\cH_p$. 
		
		}
		
	Reconstruct $H$ from $H(i,i-p)=\cH_p(i)$
		
	Calculate the eigenpair $(\lambda, v)$ of $H$ corresponding to the largest eigenvalue. 
	
	Set $x_0=\sqrt{\lambda}x$.

	\end{algorithm}

In steps $(1)$, $(2)$ and $(3)$ one has to invert a Fourier transform. If all values of the transformed vector are known, one can simply use the standard fast inverse Fourier transform to compute this, and the signal will be perfectly recovered. If one on the other hand does not know all the values (not all Gabor measurements are given), some other method has to be used, where sparsity can be employed. We have chosen Basis Pursuit \cite{chen1998atomic}. This is a standard approach in compressed sensing when looking for a sparse solution $x$ of the equation $Ax=b$. The algorithm consists of solving the following optimization problem:
	\begin{align*}
		\min \norm{x}_1 \text{ s.t. } Ax=b.
	\end{align*}
We solve this problem with CVX, a package for specifying and solving convex programs \cite{cvx}. All experiments were conducted on a Intel Core i7-3517U Processor running Windows 8. 

\renewcommand{\thesubfigure}{\Alph{subfigure}}
\begin{figure}[h!]
        \centering
        \begin{subfigure}[b]{0.48\textwidth}
     \includegraphics[width=\textwidth]{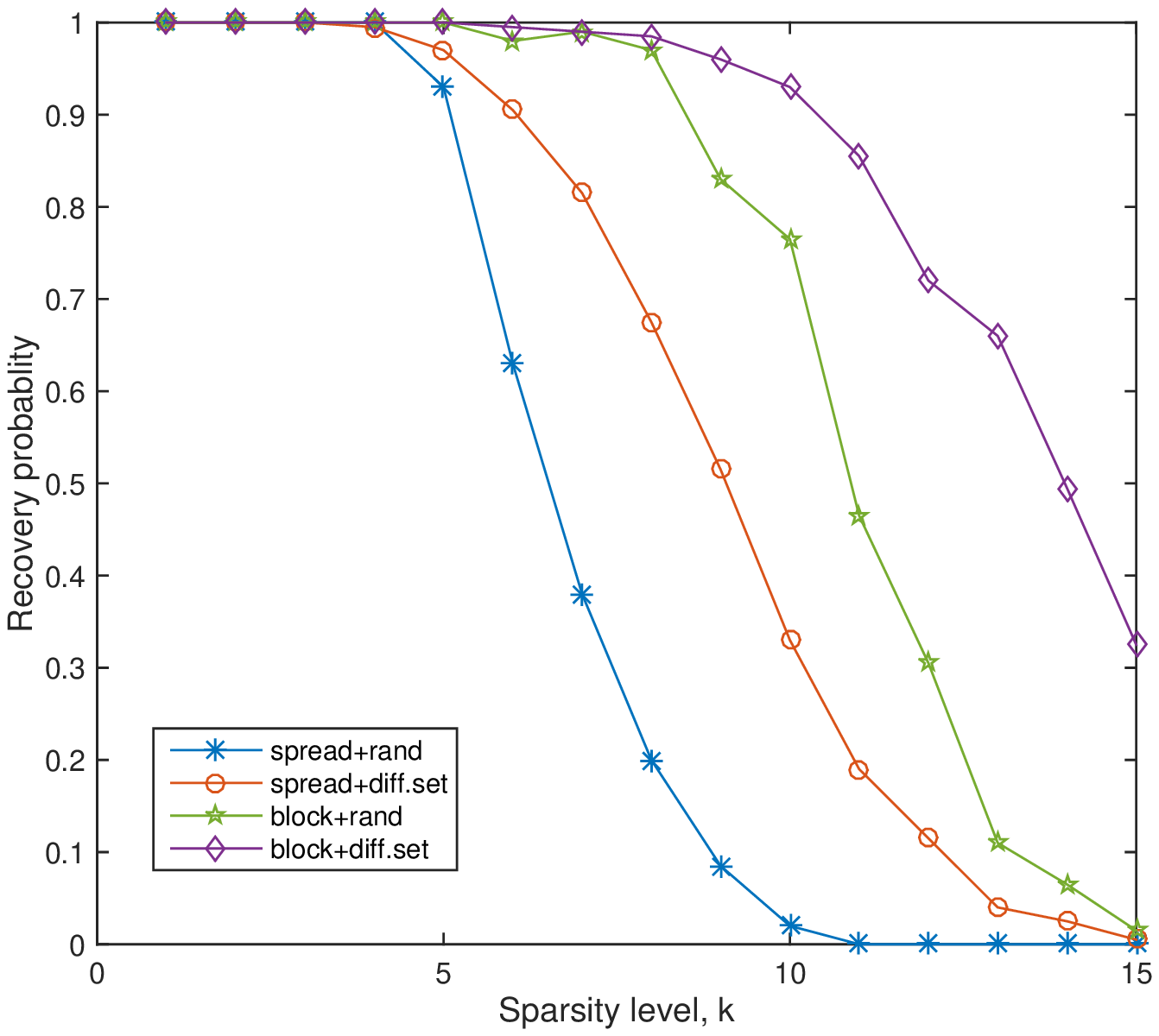}
                \caption{Partial modulations only}
                \label{fig:diffS}
        \end{subfigure}
        \begin{subfigure}[b]{0.48\textwidth}
         \includegraphics[width=\textwidth]{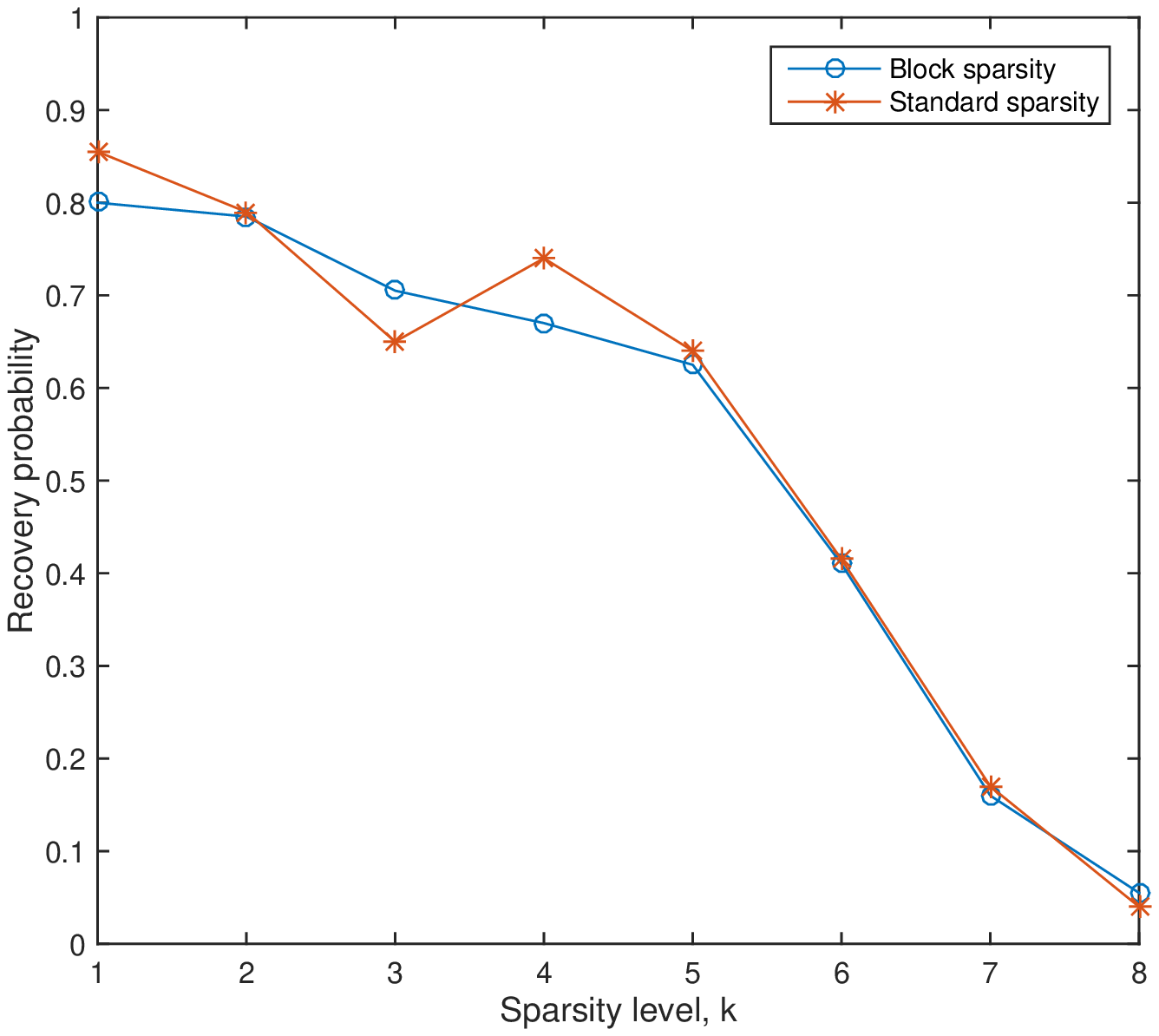}
                \caption{Partial modulations and translations}
                \label{fig:diffM}
        \end{subfigure}
       \caption{Success rate}\label{fig:experiment}
\end{figure}

We now present the results of the numerical experiments for testing Algorithm $1.$ 
In Figure \ref{fig:experiment}, we plot the success rate of recovery of sparse signals via Algorithm $1.$ We have fixed the length of the signal $N=67$ (a prime which gives $3$ $\m 4,$ as needed for difference sets of \textit{Family 1}).  We want to recover two types of signals: $k$-sparse signals, where $k$ non-zero random values are distributed on a random support, and $k$-sparse block signals,  where one random block of $k$ subsequent entries is assigned $k$ random values.

In Figure \ref{fig:diffS}, we also have chosen two different generators for the Gabor system: a complex random signal, and a characteristic function of a difference set, described in Section \ref{sec:arbitrary}. Here, we use $0.5N^2$ measurements,
namely all $N$ translations, and random $0.5N$ from the modulations. With this setup, we use $\ell_1$ minimization only in the Step $(1),$ and in $(2)$ and $(3)$ we use the fast Fourier transform. For a fixed sparsity from $1$ to $15,$ we repeat the experiment $T=200$ times, and count a trial as successful, if the normalized squared error is smaller than $10^{-2}.$

In Figure \ref{fig:diffM}, we do the same experiment, but we take partial measurements in both directions: translation and modulation. Namely, for $N=67,$ we take $0.52N$ translations, and $0.7N$ modulations at random. The generator here, as described in Proposition \ref{prop:Fourierwindow}, is a short Fourier window, with length $8$. Now, we need to use Basis Pursuit in all steps $(1),$ $(2)$ and $(3),$ which in turn leads to a  lower recovery rate. We made $T=100$ trials for every sparsity level.

\renewcommand{\thesubfigure}{\Alph{subfigure}}

\begin{figure}[h!]
        \centering
        \begin{subfigure}[b]{0.48\textwidth}
    \includegraphics[width=\textwidth]{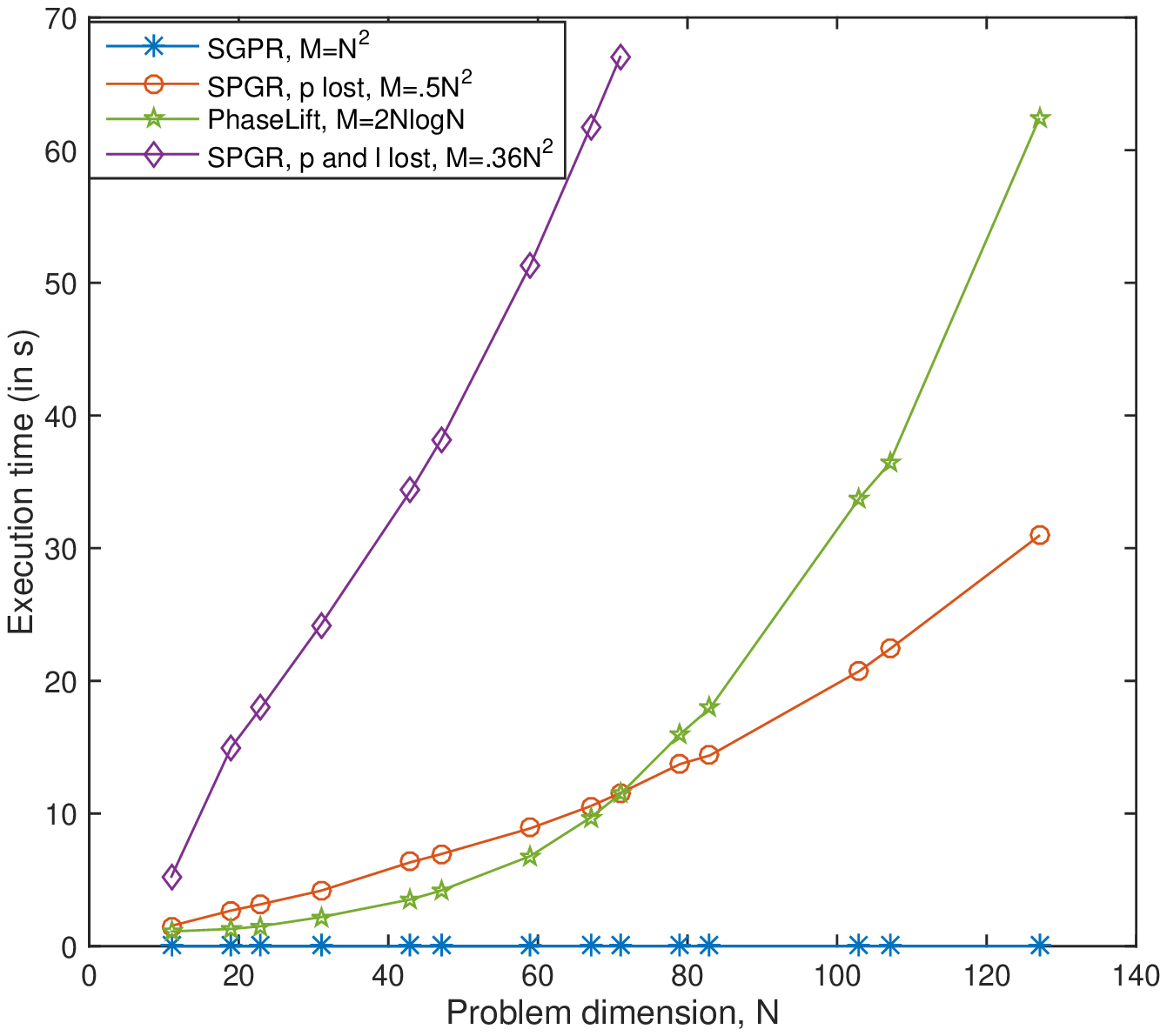}
        \caption{SGPR vs. PhaseLift}\label{fig:diffN}
        \end{subfigure}
        \begin{subfigure}[b]{0.48\textwidth}
    \includegraphics[width=\textwidth]{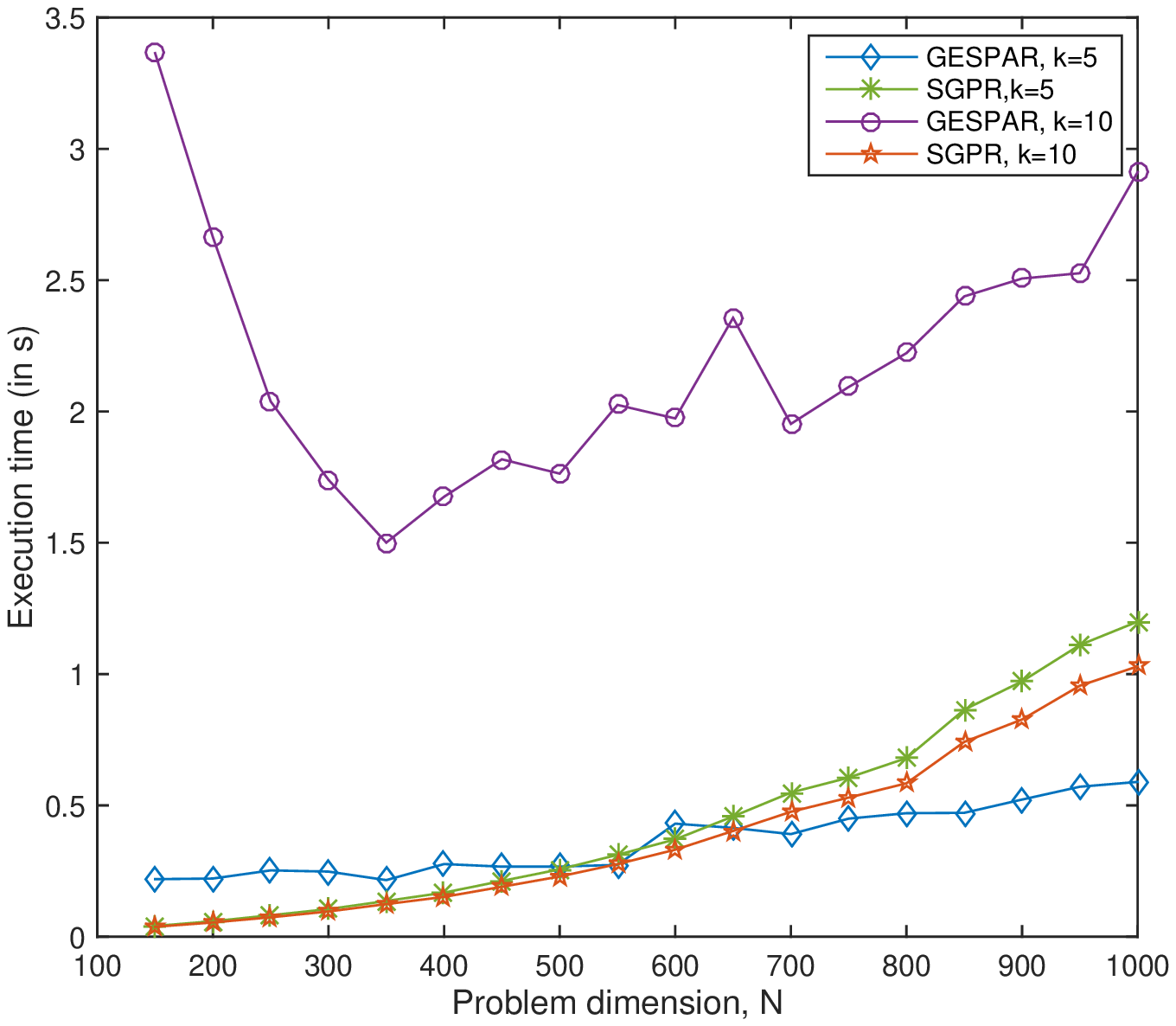}
        \caption{SGPR vs. GESPAR}\label{fig:diffNgespar}
        \end{subfigure}
   \caption{Recovery time}\label{fig:experiment}      
\end{figure}	

In Figure \ref{fig:diffN}, we test the speed of our algorithm in comparison to the PhaseLift algorithm \cite{candes2013phase}, implemented using the CVX package. We also use Gabor measurements for it, but only $2\log(N) N,$ taken at random. We plot the average execution time over $T=50$ trials, and see that as the dimension grows, our method becomes faster, although the number of measurements is much larger. Also, if we are using the full set of measurements, the time needed is incomparably smaller to both of the other methods -- since then there is no minimization problem included. In this case, also, we will always recover the signal with probability $1,$ independently of the sparsity level. 

In Figure \ref{fig:diffNgespar}, we compare the execution time of Algorithm $1$  from all $N^2$ measurements  to the GESPAR algorithm \cite{shechtman2014gespar}, a greedy algorithm for recovery of sparse signals from Fourier intensity measurements (in our experiment we use $2N$ measurements). This algorithm is very fast for high dimensions, but since it is iterative, it becomes slower as the sparsity increases for a fixed dimension of the signal. We illustrate this behavior in Figure \ref{fig:diffNgespar}, where for every dimension, we measure the average time of recovery of signals with sparsity $k=5$ and $k=10.$  We see that the GESPAR algorithm is faster, when we want to recover a signal which has only $5$ nonzeros, but if this number is larger, our algorithm becomes faster than the GESPAR, since it does not strongly depend on the sparsity level.

We would like to mention that our algorithm for all $N^2$ measurements is also stable to additive noise in the measurements. This follows from Theorem \ref{th:framebounds} and can be intuitively explained by the fact that the only troublesome part is the division in Step $(3).$ If the generator $g$ is such, that the values $\langle g, g_{p,l} \rangle$ are bounded away from zero, one can guarantee robustness to noise. For the recovery from less than $N^2$ measurements, we leave the detailed investigation on this matter for future work.

\end{section}

\section*{Acknowledgment}
The authors would like to thank Gitta Kutyniok and Peter Jung for fruitful discussions and remarks, and Martin Sch\"afer, who assisted in the proof-reading of the manuscript. 
I.~B. acknowledges support by the Berlin Mathematical School. A.~F. acknowledges support by Deutsche 
Forschungsgemeinschaft (DFG) Grant KU 1446/18-1 and by the Deutscher Akademischer Austausch Dienst (DAAD).

\bibliography{bibliographyPR}{}  
\bibliographystyle{abbrv}  

\end{document}